\documentclass[a4paper, reqno, 11pt]{amsart}
\makeatletter

\@addtoreset{equation}{section}
\makeatother
\usepackage{setspace}
\usepackage{xcolor}
\usepackage{float}
\usepackage{amsmath}
\usepackage{amssymb}
\usepackage{latexsym}
\usepackage{amsthm}
\usepackage{hyperref}
\newtheorem{Thm}{Theorem}[section]

\newtheorem{Lem}[Thm]{Lemma}
\newtheorem{Prop}[Thm]{Proposition}

\theoremstyle{definition}

\newtheorem{Rem}[Thm]{Remark}

\begin{document}

\title[Bahadur efficiency of  MLE and one-step estimator]{Bahadur efficiency of the maximum likelihood estimator  and one-step estimator for quasi-arithmetic means of the Cauchy distribution}

\author[Y.Akaoka]{Yuichi Akaoka}
\address{Department of Mathematics, Faculty of Science, Shinshu University}
\curraddr{Gunma bank}
\email{18ss101b@gmail.com}

\author[K. Okamura]{Kazuki Okamura}
\address{Department of Mathematics, Faculty of Science, Shizuoka University}
\email{okamura.kazuki@shizuoka.ac.jp}

\author[Y. Otobe]{Yoshiki Otobe}
\address{Department of Mathematics, Faculty of Science, Shinshu University}
\email{otobe@math.shinshu-u.ac.jp}

\date{\today}
\dedicatory{}

\maketitle

\begin{abstract}
Some quasi-arithmetic means of random variables easily give unbiased strongly consistent closed-form estimators of the joint of the location and scale parameters of the Cauchy distribution. 
The one-step estimators of those quasi-arithmetic means of the Cauchy distribution are considered. 
We establish the Bahadur efficiency of the maximum likelihood estimator and the one-step estimators. 
We also show that the rate of the convergence of the mean-squared errors  achieves the Cram\'er-Rao bound. 
Our results are also applicable to the circular Cauchy distribution. 
\end{abstract}


\section{Introduction}

In the parameter estimation of the location $\mu \in \mathbb{R}$ and the scale $\sigma > 0$ of the Cauchy distribution, 
it is difficult to balance efficiency with computational difficulty. 
So far, various approaches have been taken. 
The maximal likelihood estimation (MLE) has been considered by \cite{Haas1970, Copas1975, Ferguson1978, Gabrielsen1982, Reeds1985, Saleh1985, Bai1987,  Vaughan1992, McCullagh1992, McCullagh1993, McCullagh1996, Matsui2005}.  
The order statistics, which includes analysis for central values and quantiles,  is used in \cite{Ogawa1962, Ogawa1962a, Rothenberg1964, Barnett1966, Bloch1966, Chan1970, Balmer1974, Cane1974, Rublik2001, Zhang2009, Kravchuk2012}.  
Other approaches are taken by \cite{Howlader1988, Higgins1977, Higgins1978, Boos1981, Guertler2000, Besbeas2001, Onen2001, Kravchuk2005, CohenFreue2007}.  
Results obtained before 1994 are thoroughly surveyed in the book by Johnson-Kotz-Balakrishnan \cite[Chapter 16]{Johnson1994}. 

In \cite{Akaoka2021-1}, 
the authors suggest new estimators of the parameters of the Cauchy distribution in the case that neither the location $\mu$ nor the scale $\sigma$ is known by dealing with quasi-arithmetic means of independent and identically distributed (i.i.d.) random variables. 
For  the parameter estimation of the Cauchy distribution, 
some of the quasi-arithmetic means of a sample have closed-forms,  and are unbiased and strongly consistent, under McCullagh's parametrization \cite{McCullagh1993, McCullagh1996}. 
Those quasi-arithmetic means are easy to construct and analyze rigorously.  
Indeed, in \cite{Akaoka2021-3}, the authors construct confidence discs for $\mu + \sigma i$ without numerical analysis. 

Considering the one-step estimators of $\sqrt{n}$-consistent estimators is a useful way to obtain efficient estimators. 
It is well-known that under some conditions, they achieve the Cram\'er-Rao bound via the central limit theorem for the one-step estimators (See \cite[Section 7.3]{Lehmann1999}).    
Our one-step estimators are simple and easy to calculate, according that the initial estimators, which are quasi-arithmetic means here, are simple and easy to calculate. 
We also show that our one-step estimators and the MLE are efficient in the Bahadur sense, which concerns  large deviation estimates.  
Our proof of the case of the MLE  depends on Arcones \cite{Arcones2006}, Shen \cite{Shen2001} and the explicit formula of the Kullback-Leibler divergence between the Cauchy distributions recently obtained by Chyzak-Nielsen \cite{Chyzak2019}. 
We also show that the rate of the convergence of the mean-squared errors achieves the Cram\'er-Rao bound. 
Our results are also applicable to the circular Cauchy distribution, which are closely connected with the Cauchy distribution via the M\"obius transformations.

This paper is organized as follows. 
In Section 2, we give some preliminary results for tail estimates of quasi-arithmetic means. 
In Section 3, we establish the Bahadur efficiency of the MLE. 
In Section 4, we establish the Bahadur efficiency and the rate of the convergence of the mean-squared errors of the one-step estimators of the quasi-arithemtic means, which are the main results of this paper. 
In Section 5, we give an application of  parameter estimation for the circular Cauchy distribution. 
Section 6 is devoted to proofs of the results in Sections 2, 3, 4 and 5. 
In Appendix, we give numerical computations for the mean-squared errors of the one-step estimators in Section 4. 

\subsection{Framework}

Let $\mathbb H$ be the upper-half plane and $\overline{\mathbb H} = \mathbb{H} \cup \mathbb R$. 
Let $i$ be the imaginary unit. 
We often denote $\theta := \mu + \sigma i$. 
For $\theta = \mu + \sigma i$, we let 
\[ P_{\theta}(dx) = p(x; \theta) dx, \ \ p(x; \theta) = p(x; (\mu, \sigma)) := \frac{\sigma}{\pi} \frac{1}{(x - \mu)^2 + \sigma^2}. \]
For $\alpha \in \mathbb C$, we denote its complex conjugate by $\overline{\alpha}$. 

Let $(X_i)_i$ be an i.i.d. sequence of random variables with distribution $C(\mu, \sigma)$ on a probability space. 
We often denote the distribution following $P_{\theta}$ by $C(\theta)$.   

We deal with the quasi-arithmetic means of the i.i.d. random variables.   
This has the form that 
\[ f^{-1} \left( \frac{1}{n} \sum_{j=1}^{n} f(X_j) \right), \]
where $U$ is a domain containing $\overline{\mathbb H}  \setminus \{\alpha\}$ for some $\alpha \in \{x+yi : y \le 0\}$ and $f : U \to \mathbb C$ is a continuous injective holomorphic  function such that $f\left(\overline{\mathbb H}  \setminus \{\alpha\}\right)$ is convex, $\lim_{z \to \alpha} (z-\alpha) f(z) = 0$ and furthermore $f$ has sublinear growth at infinity, specifically, $\lim_{|z| \to \infty} f(z)/|z| = 0$. 

We remark that $f$ is not only in the $C^{\infty}$ class but also holomorphic on a domain containing $\overline{\mathbb H}  \setminus \{\alpha\}$. 
By the residue theorem,  
\begin{equation}\label{residue}
E[f(X_1)] = f(\theta). 
\end{equation}
If $\textup{Im}(\alpha) < 0$, then, $\overline{\mathbb H}  \setminus \{\alpha\} = \overline{\mathbb H}$. 
On the other hand, if $\textup{Im}(\alpha) = 0$, then, $\overline{\mathbb H}  \setminus \{\alpha\} \subsetneq \overline{\mathbb H}$. 
See \cite{Akaoka2021-1} for more details. 

In this paper we additionally assume that there exists a constant $\lambda > 0$ such that 
\begin{equation}\label{exp-moment}
E\left[\exp\left(\lambda \left| f(X_1) \right| \right) \right] < +\infty. 
\end{equation} 

\cite{Akaoka2021-1} deals with quasi-arithmetic means with its generators of the form 
\[ f(x) = \log(x+\alpha), \ \alpha \in \overline{\mathbb{H}}, \]
or 
\[ f(x) =  \frac{1}{x+\alpha}, \ \ \ \alpha \in \mathbb{H},  \]
each of which corresponds to the geometric mean and a modification of the harmonic mean.  
We remark that $\alpha \in \mathbb{R}$ is allowed 
when we deal with the case that $f(x) = \log (x+\alpha)$. 
\eqref{residue} and \eqref{exp-moment} hold for the case that $f(x) = \log (x+\alpha), \ \alpha \in \overline{\mathbb{H}}$ and that $f(x) = 1/(x + \alpha), \ \alpha \in \mathbb{H}$.  
We remark that the quasi-arithmetic means with generator $f(x) =  1/(x + \alpha)$ are identical with the  quasi-arithmetic means with generator 
$$f(x) = \frac{x + \overline{\alpha}}{x  + \alpha} = \frac{x - \beta}{x - \overline{\beta}}, \ \ \beta = -\overline{\alpha}.$$  
See \cite[Example 1.2 (ii)]{Akaoka2021-1}  for more details.

\section{Large deviations for quasi-arithmetic means}

We first give a decay rate of quasi-arithmetic means. 
We assume that  \eqref{residue} and \eqref{exp-moment} hold.  

\begin{Thm}\label{LDP-chernoff}
There exist positive constants $c_1$ and $c_2$ such that 
for every $\epsilon \in (0, \sigma)$ and every $n \ge 2$, 
\[ P\left( \left|  f^{-1} \left( \frac{1}{n} \sum_{j=1}^{n} f(X_j) \right)  - \theta \right| > \epsilon\right) \le c_1 \exp\left(-c_2  n \epsilon^2\right). \]
\end{Thm}

This also plays a crucial role in the proof of the Bahadur efficiency of the one-step estimator with initial estimator.   

By the contraction principle, 
we have the following:
\begin{Thm}
The quasi-arithmetic means $ f^{-1} \left( \frac{1}{n} \sum_{j=1}^{n} f(X_j) \right)$ satisfy the large deviation principle with rate function 
\begin{multline*}
I(z) 
= \sup_{\lambda_1, \lambda_2 \in \mathbb{R}} \biggl(\lambda_1 \textup{Re}(f(z)) + \lambda_2 \textup{Im}(f(z)) \\ - \log E\left[ \exp\left(  \lambda_1 \textup{Re}(f(z)) + \lambda_2 \textup{Im}(f(z)) \right) \right]\biggr),    
\end{multline*}
where $z \in \mathbb{H}$. 
\end{Thm}

See Dembo-Zeitouni \cite{Dembo2010} for the basic terminologies and results of the theory of large deviations. 

Under the consideration of \cite[Eq.(1.5)]{Akaoka2021-1}, 
we conjecture that 
\[ \lim_{\epsilon \to +0} \frac{1}{\epsilon^2}  \lim_{n \to \infty} \frac{1}{n} P\left( \left|  f^{-1} \left( \frac{1}{n} \sum_{j=1}^{n} f(X_j) \right)  - \theta \right| > \epsilon\right) = -\frac{\left| f^{\prime} \left( \theta \right) \right|^2}{\textup{Var} (f(X_1))}. \]
The right hand side is called an {\it inaccuracy rate}. 
The quasi-arithmetic mean $ f^{-1} \left( \frac{1}{n} \sum_{j=1}^{n} f(X_j) \right)$ can be regarded as an M-estimator.  
Let $\phi(x, \theta) := f(x) - f(\theta)$.  
Then, $  \theta =  f^{-1} \left( \frac{1}{n} \sum_{j=1}^{n} f(X_j) \right)$
if and only if $ \sum_{j=1}^{n} \phi(X_j, \theta) =  0$. 

The inaccuracy rates for one-dimensional M-estimators are considered by Jure{\v{c}}kov{\'{a}} and Kallenberg \cite{Jureckova1987}. 
Multidimensional large deviation principles for M-estimators are considered by \cite{Arcones2006}.  
\cite[Theorem 3.5]{Arcones2006} is not applicable to the case the both of the location and the scale are unknown.  
We do not pursue further properties for large deviations for complex-valued M-estimators here. 

\section{Bahadur efficiency of  the maximum likelihood estimator}

Let $\hat{\theta}_n$ be the maximal likelihood estimator (MLE) of $\theta = \mu + \sigma i$.  
Copas \cite{Copas1975}  showed that the joint likelihood function for the location and scale parameters of the Cauchy distribution is unimodal.  
Let 
\[ h(x,t) := \frac{x-t}{x- \overline{t}}, \ \ x \in \overline{\mathbb H}, t \in \mathbb{H}. \]
The MLE $\hat{\theta}_n = \hat{\theta}_n (X_1, \dots, X_n)$ is a unique solution of the following  equation for $\alpha$: 
\[ \sum_{j =1}^{n}  h(X_j, \alpha) =  \sum_{j =1}^{n} \frac{X_j - \alpha}{X_j - \overline{\alpha}} = 0. \]
This is the likelihood equation in the complex form.   
See \cite[Corollary 2.8]{Okamura2021} for more details and \cite[Proposition 2.2]{Okamura2021} for another expression for  the likelihood equation in the complex form. 
For $n \le 4$, the solution of the likelihood equation has a closed form (Ferguson \cite{Ferguson1978}).  
However, for $n \ge 5$, we do not have any algebraic closed-form formula (\cite{Okamura2021}). 

Let the Kullback-Leibler divergence be 
\[ K(P_{\theta^{\prime}}| P_{\theta}) := \int_{\mathbb R} p(x; \theta^{\prime}) \log \frac{p(x; \theta^{\prime})}{p(x; \theta)} dx, \ \ \theta \in \mathbb H. \]

\begin{Prop}[Chyzak-Nielsen   {\cite{Chyzak2019}}]\label{CN}
For $\theta, \theta^{\prime} \in \mathbb H$, 
\begin{equation}\label{KL}
K(P_{\theta^{\prime}}| P_{\theta}) = \log\left(1 + \frac{| \theta - \theta^{\prime}|^2}{4 \textup{Im}(\theta) \textup{Im}(\theta^{\prime})}\right). 
\end{equation} 
\end{Prop}

We remark that a version of \eqref{KL} already appears in \cite[Eq.(18)]{McCullagh1996}.  
See also the discussions around \eqref{MC14} and \eqref{rate-analysis} below. 
Let 
\[ SL(2, \mathbb R) := \left\{  \begin{pmatrix} a & b \\ c & d \end{pmatrix} : a, b, c, d \in \mathbb{R}, ad - bc = 1 \right\}. \]
The quantity $ \frac{| \theta - \theta^{\prime}|^2}{4 \textup{Im}(\theta) \textup{Im}(\theta^{\prime})}$ in \eqref{KL} is the maximal invariant for the action of $SL(2, \mathbb{R})$ to $\mathbb H \times \mathbb H$ defined by 
\[ A \cdot (z,w) = \left(\frac{az+b}{cz+d},  \frac{aw+b}{cw+d}\right), \ \ A = \begin{pmatrix} a & b \\ c & d \end{pmatrix} \in SL(2, \mathbb{R}), \ \ z,w \in \mathbb{H}. \] 
See also \cite{McCullagh1993, McCullagh1996}. 
In particular we see that 
\[ K(P_{A \cdot \theta^{\prime}}| P_{A \cdot \theta}) =K(P_{\theta^{\prime}}| P_{\theta}). \]

Let 
\[ b(\epsilon, \theta) := \inf\left\{K(P_{\theta^{\prime}}| P_{\theta}) : \theta^{\prime} \in \mathbb{H},  |\theta^{\prime} - \theta| > \epsilon \right\}.  \]
Then, by \eqref{KL}, 
\begin{equation}\label{b-formula} 
b(\epsilon, \theta) = \log\left(1+ \frac{\epsilon^2}{4 \textup{Im}(\theta) (\textup{Im}(\theta) + \epsilon)} \right). 
\end{equation} 
This does not depend on the location $\mu$. 
In particular,  
\begin{equation}\label{b-order}  
\lim_{\epsilon \to +0} \frac{b(\epsilon, \theta)}{\epsilon^2} = \frac{1}{4 \textup{Im}(\theta)^2}.  
\end{equation} 

Bahadur, Gupta and Zabell  \cite{Bahadur1980} showed that for every consistent estimator $T_n$ for each $\theta$, 
\begin{equation*}
 \liminf_{n \to \infty} \frac{\log P(|T_n - \theta| >  \epsilon)}{n} \ge -b(\epsilon, \theta). 
\end{equation*}  
See \cite[Eq. (4)]{Arcones2006} or \cite[Proposition 1]{Shen2001}. 
It is a natural question whether we can show that 
\begin{equation}\label{Bahadur-equal} 
\liminf_{n \to \infty} \frac{\log P(|T_n - \theta| >  \epsilon)}{n} = -b(\epsilon, \theta), 
\end{equation} 
in the case that $(T_n)_n$ is the MLE.  
Kester and Kallenberg \cite{Kester1986} gave an example of the MLE which does  not satisfy  \eqref{Bahadur-equal}.

The following theorem along with the above estimate indicates that the MLE is best in the sense of Bahadur efficiency. 

\begin{Thm}[Bahadur efficiency of the MLE]\label{Bahadur-MLE}
\[ \limsup_{\epsilon \to +0} \limsup_{n \to \infty} \frac{\log P(|\hat{\theta}_n - \theta| >  \epsilon)}{n b(\epsilon, \theta)} = -1. \] 
\end{Thm}

Bai-Fu \cite{Bai1987} considered the MLE of location $\mu$ in the case that the scale parameter $\sigma$ is known. 
In that case the likelihood equation is given by 
\begin{equation}\label{multiple}
\sum_{j=1}^{n} \frac{X_j - \mu}{(X_j - \mu)^2 + \sigma^2} = 0 
\end{equation} 
and this equation for $\mu$ has multiple solutions (Reeds \cite{Reeds1985}), so from this equation itself, 
we cannot see which $\mu$ of the solutions gives the maximum of the log-likelihood.  
 \cite[Eq. (2.12) and (2.13)]{Bai1987} shows that 
\[ \lim_{n \to \infty} \frac{\log P(|\hat{\theta}_n - \mu| >  \epsilon)}{n} = -\beta(\epsilon),  \]
where 
\[ \beta(\epsilon) = \frac{ \epsilon^2}{4 \sigma^2} \left(1 + O(\sqrt{\epsilon})\right). \]
However, by \cite[Theorem 3.2]{Arcones2006}, it holds that $\beta(\epsilon) \ne b(\epsilon, \sigma i)$.  

Our strategy of the proof of  Theorem \ref{Bahadur-MLE} is to firstly establish that $P(|\hat{\theta}_n - \theta| > \epsilon)$ decays exponentially for every $\epsilon > 0$, and then to adopt Shen \cite[Theorem 3]{Shen2001}. 
\cite{McCullagh1996} obtains for an asymptotic pointwise lower bound for the density function $p_n (\chi)$ of the MLE $\hat{\theta}_n$ with respect to an invariant measure on $\mathbb{H}$ for the action of the special linear group $SL_2 (\mathbb{R})$, under the assumption of existence of the continuous density function. 
However it might not lead any estimates for the upper bound of $P\left(|\hat{\theta}_n - \theta| > \epsilon\right)$. 

We still conjecture that the rate function in \cite[Theorem 3.8]{Arcones2006} is correct, although it is not applicable to the Cauchy distribution in its form. 
The rate function is defined by 
\[ I_{\theta}(\gamma) :=  -\inf_{\lambda_1, \lambda_2 \in \mathbb{R}} \log E^{\theta}\left[\exp\left(\lambda_1 \textup{Re}\left( h(x, \gamma) \right) + \lambda_2 \textup{Im}\left( h(x, \gamma) \right) \right)\right], \] 
where $\gamma, \theta \in \mathbb H$ and $E^{\theta}$ denotes the expectation with respect to $P_{\theta}$.  
Let $\gamma = \gamma_1 + i \gamma_2$.
This might be the limit of $-\log p_n (\chi) /n$, 
where $\chi := \dfrac{(\gamma_1 - \mu)^2 + (\gamma_2 - \sigma)^2}{4 \gamma_2 \sigma}$.

\cite[p.801]{McCullagh1996} states that 
\begin{equation}\label{MC14} 
\liminf_{n \to \infty} \frac{\log p_n \left(\chi \right)}{n} \ge -\log\left(1+  \frac{(\gamma_1 - \mu)^2 + (\gamma_2 - \sigma)^2}{4 \gamma_2 \sigma}\right) = -K\left(P_{\gamma} | P_{\theta}\right). 
\end{equation}

Since it holds that for Borel measureble set $A$, 
\[ P_{\theta}(\hat{\theta}_n \in A) = \iint_A \frac{1}{4\pi t_2^2} p_n \left( \frac{(t_1 - \mu)^2 + (t_2 - \sigma)^2}{4 t_2 \sigma} \right) dt_1 dt_2, \]
we would naturally expect that for sufficiently small $\epsilon > 0$, 
\begin{equation}\label{rate-analysis} 
\frac{\log p_{n} \left( \chi \right) }{n} \approx \frac{\log P_{\theta}(\hat{\theta}_n \in B(\gamma, \epsilon))}{n} \approx -I_{\theta}(\gamma), \ \ n \to \infty.  
\end{equation}

By \cite[Eq.(33)]{Arcones2006}, 
we see that $K(P_{\gamma} | P_{\theta}) \ge I_{\theta}(\gamma)$, which is consistent with \eqref{MC14} and \eqref{rate-analysis}. 
If the conjecture \cite[Eq.(18)]{McCullagh1996} is true, then, $K(P_{\gamma} | P_{\theta}) = I_{\theta}(\gamma)$. 
Furthermore, we can show that 
$ I_{A \cdot \theta}(A \cdot \gamma) =  I_{\theta}(\gamma)$ for every $A \in SL(2, \mathbb R)$ 
by the same technique as in the proof of \cite[Lemma 2.3]{Okamura2020a}.

\section{Bahadur efficiency of  the one-step estimator of quasi-arithmetic means}

For ease of notation, we let 
\begin{equation}\label{def-Yn}
Y_n := f^{-1} \left( \frac{1}{n} \sum_{i=1}^{n} f(X_i)\right). 
\end{equation}


Let 
\[ \Psi(x; \theta) = \Psi(x; (\mu, \sigma)) := \begin{pmatrix} \frac{\partial}{\partial \mu} \log p(x; (\mu, \sigma)) \\ \frac{\partial}{\partial \sigma} \log p(x; (\mu, \sigma)) \end{pmatrix}.  \]
Let $I_n (\theta) = I_n (\mu, \sigma):= \frac{n}{2\sigma^2} I_2$, that is, the Fisher information matrix for the Cauchy sample of size $n$. 
Now we regard $Y_n$ as an $\mathbb R^2$-valued random variable. 
Then, a version of the one-step estimator of $Y_n$ is given by 
\[ Z_n := Y_n - I_n (Y_n)^{-1} \sum_{j=1}^{n} \Psi(X_j; Y_n).  \]

Now we rewrite this in terms of the complex parametrization. 
By recalling the definition of $h$, 
\[  \Psi(x; \theta) = \frac{1}{\textup{Im}(\theta) |x-\theta|^2}  \begin{pmatrix} \textup{Im}((x-\theta)^2) \\ \textup{Re}((x-\theta)^2)\end{pmatrix} = \frac{1}{\textup{Im}(\theta)}  \begin{pmatrix} \textup{Im}(h(x, \theta)) \\  \textup{Re}(h(x, \theta)) \end{pmatrix}. \]
Now we see that 
\begin{equation}\label{def-Zn} 
Z_n = Y_n + \frac{2 \textup{Im}(Y_n) i}{n} \sum_{j =1}^{n} h(X_j, Y_n)  = Y_n + \frac{2 \textup{Im}(Y_n) i}{n} \sum_{j =1}^{n} \frac{X_j - Y_n}{X_j - \overline{Y_n}}. 
\end{equation} 

The main feature is that the initial estimators $(Y_n)_n$ and the one-step estimators $(Z_n)_n$ have closed-form. 
Therefore it is easy to compute and do not need numerical approximations, which is contrary to the MLE $(\hat{\theta}_n)_n$. 
If $\textup{Var}(f(X_1)) < +\infty$, 
then, by the central limit theorem for $(Y_n)$, which is stated in \cite[Theorem 1.5]{Akaoka2021-1}, 
$(Y_n)_n$ is a $\sqrt{n}$-consistent estimator. 
Then it holds that 
\begin{equation}\label{Z-Fisher} 
\sqrt{n} (Z_n - \theta) \Rightarrow N\left(0, \ \frac{1}{2\sigma^2} I_2  \right), \ n \to \infty, 
\end{equation}
where $\Rightarrow$ means the convergence in distribution, $I_2$ denotes the identity matrix of degree 2, and, $N(\cdot, \cdot)$ is the 2-dimensional normal distribution.     
See Lehmann \cite[Section 7.3]{Lehmann1999} for details.

The one-step estimators are easily obtained by efficient estimators. 
In general, the difficulty arising from multiple roots of the likelihood equations, as in \eqref{multiple} above, is overcome by a one-step estimator which only requires a $\sqrt{n}$-consistent estimator in the initial point.   
The following is the main result of this paper. 
\begin{Thm}[Bahadur efficiency of  the one-step estimator]\label{main}
If \eqref{exp-moment} holds, then, 
\[ \limsup_{\epsilon \to +0} \limsup_{n \to \infty} \frac{\log P\left(|Z_n - \theta| >  \epsilon \right)}{n b(\epsilon, \theta)} \le -1. \] 
\end{Thm}

The proof is done by estimating $Z_n - \hat{\theta}_n$.  
This is different from the consideration of one-step estimators by \cite{Janssen1985}  in the one-dimensional case. 
We might be able to consider it  for not only the Cauchy distribution but also other distributions, 
if the tail of the starting point estimator decays exponentially fast. 

Let $(T_n)_n$ be a sequence of complex-valued {\it unbiased} estimators of $\mu + \sigma i$, where $T_n$  is an unbiased estimator for samples of size $n$. 
Then, $\textup{Re}(T_n)$ and $\textup{Im}(T_n)$ are unbiased estimators of $\mu$ and $\sigma$ respectively. 
By the Cram\'er-Rao inequality, 
\begin{equation}\label{CR}
n \textup{Var}(T_n)  \ge 4 \textup{Im}(\theta)^2.   
\end{equation}  

The following theorem shows $(Z_n)_n$ achieves the lower bound of \eqref{CR}, although it may not be unbiased. 

\begin{Thm}[Variance asymptotics for the one-step estimator]\label{main-var}
If 
$$E\left[ \left|f(X_1)\right|^2 \right] < +\infty$$ 
and 
\begin{equation}\label{main-var-ass} 
\lim_{n \to \infty} n\textup{Var}\left( Y_n \right)  = \frac{\textup{Var}(f(X_1))}{\left|f^{\prime}(\theta) \right|^2}, 
\end{equation} 
 then, 
 \[ \lim_{n \to \infty} n E\left[ \left| Z_n - \theta \right|^2 \right] = 4 \, \textup{Im}(\theta)^2, \] 
 in particular, 
\[ \lim_{n \to \infty} n\textup{Var}\left( Z_n \right)  \le 4 \, \textup{Im}(\theta)^2. \] 
\end{Thm}

\eqref{main-var-ass} is identical with \cite[Eq.(1.5)]{Akaoka2021-1} and is verified for the case that $f(x) = \log (x+\alpha), \ \alpha \in \overline{\mathbb{H}}$ and $f(x) = \dfrac{x - \alpha}{x - \overline{\alpha}}, \ \alpha \in \mathbb{H}$ in \cite[Theorems 4.2 and 4.4]{Akaoka2021-1}. 
The proof is done by establishing the uniform integrability for $\left\{\sqrt{n} (Z_n - \theta) \right\}_n$. 
Numerical computations for $n E\left[ \left| Z_n - \theta \right|^2 \right]/\textup{Im}(\theta)^2$ and related discussions are given in Appendix. 

\section{Parameter estimation for the circular Cauchy distribution}

In this section, we apply the results in Section 4 to parameter estimations for the circular Cauchy distribution. 
The circular Cauchy distribution, also known as the wrapped Cauchy distribution, appears in the area of directional statistics. 
It is regarded as a distribution on the unit circle. 
It is connected with the Cauchy distribution via the M\"obius transformations.  
Such connection is considered by McCullagh \cite{McCullagh1992, McCullagh1996}. 
Recently, in Kato and McCullagh \cite{Kato2020}, 
an extension to the high-dimensional sphere is investigated in terms of the M\"obius transformations. 
The maximum likelihood estimation of the circular Cauchy distribution is attributed to that of  the Cauchy distribution.  
Due to the connection, we can apply our results in Section 4 to the circular Cauchy distribution in a simple manner. 

Let $\mathbb D := \{z \in \mathbb{C} : |z| < 1\}$. 
The circular Cauchy distribution $P^{\textup{cc}}_w$ with parameter $w \in \mathbb{D}$ is the continuous distribution on $[0, 2\pi)$  
with density function 
\[ p^{\textup{cc}}(x; w) := \frac{1}{2\pi} \frac{1 - |w|^2}{|\exp(ix) - w|^2}, \]
where we have used McCullagh's expression \cite{McCullagh1996}. 
We estimate the parameter $w$.   

In this section we let $\phi_{\alpha} : \mathbb{H} \to \mathbb{D}$ be the function defined by 
\[ \phi_{\alpha}(z) := h(z, \alpha) = \frac{z - \alpha}{z - \overline{\alpha}}. \]
Then $\phi_{\alpha}$ is a bijection and its inverse is given by 
\[ \phi_{\alpha}^{-1}(w) = \frac{\alpha - \overline{\alpha}w}{1 - w}, \ w \in \mathbb{D}. \]

Let $\widetilde X_n, n \ge 1,$ be i.i.d. random variables following the circular Cauchy distribution $P^{\textup{cc}}_{w}$. 
Let $X_n := \phi_{\alpha}^{-1}(\exp(i \widetilde X_n))$. 
They are i.i.d. random variables following the Cauchy distribution with parameter $\phi_{\alpha}^{-1}(w)$. 
Let $Y_n$ be the quasi-arithmetic mean of $X_n$ with a generator $f$ defined by \eqref{def-Yn}. 
Let $Z_n$ be the one-step estimator of $Y_n$ defined by \eqref{def-Zn}.  

Let $W_n := \phi_{\alpha}(Z_n)$. 
Then, by straightforward computations, 
\begin{Lem}
For $w_1, w_2 \in \mathbb{D}$ and $\alpha \in \mathbb{H}$, 
\[ K(P^{\textup{cc}}_{w_1}|P^{\textup{cc}}_{w_2}) = K\left(P_{\phi_{\alpha}^{-1}(w_1)} | P_{\phi_{\alpha}^{-1}(w_2)}\right) = \log\left(1+\frac{|w_1 - w_2|^2}{(1 - |w_1|^2)(1 - |w_2|^2)}\right). \] 
\end{Lem}

The following are main results in this section, which correspond to Theorems \ref{main} and \ref{main-var} respectively. 

\begin{Thm}\label{Bahadur-cc}
\[ \limsup_{\epsilon \to +0} \limsup_{n \to \infty} \frac{\log P(|W_n - w| > \epsilon)}{n \widetilde b(\epsilon, w)} \le -1,   \]
where we let 
\[ \widetilde b(\epsilon, w) := \inf\left\{K(P^{\textup{cc}}_{w}|P^{\textup{cc}}_{w^{\prime}}) : |w - w^{\prime}| > \epsilon \right\}. \]
\end{Thm}

\begin{Thm}\label{Var-CR-cc}
\[ \limsup_{n \to \infty} n \textup{Var}(W_n) \le  \lim_{n \to \infty} n E\left[|W_n - w|^2\right] = \left(1 - |w|^2 \right)^2. \]
\end{Thm}

The estimator $W_n$ has a simple closed-form, contrary to the MLE for the circular Cauchy distribution. 
These assertions are somewhat easy consequences of Theorems \ref{main} and \ref{main-var}.  

\begin{Rem}
\cite[Section 4.3]{Kato2020} deals with a one-step estimator different from ours, which is applicable to the higher dimensional case. 
\end{Rem}

\section{Proofs}

\subsection{Proofs of assertions in Section 2}

\begin{proof}[Proof of Theorem \ref{LDP-chernoff}]
Since $f^{-1}$ is injective and holomorphic on an open neighborhood of $f(\theta)$, we see that 
there exists  a constant $c_3  = c_3 (\mu, \sigma) > 0$ such that for every $\epsilon \in (0, \sigma)$  
and every $n \ge 2$, 
\[P\left(\left| f^{-1} \left( \frac{1}{n} \sum_{j=1}^{n} f(X_j)  \right) - \theta \right| > \epsilon \right)  \le P\left( \left| \frac{1}{n} \sum_{i=1}^{n} f(X_i) - f(\theta) \right| > c_3 \epsilon \right)\]
\[ \le P\left( \left| \frac{1}{n} \sum_{i=1}^{n} \textup{Re}(f(X_i)) - \textup{Re}(f(\theta)) \right| > \frac{c_3 \epsilon}{2} \right) + P\left( \left| \frac{1}{n} \sum_{i=1}^{n} \textup{Im}(f(X_i)) - \textup{Im}(f(\theta)) \right| > \frac{c_3 \epsilon}{2} \right). \]

By \eqref{residue} and \eqref{exp-moment}, 
we can apply the Cram\'er-Chernoff method and have that 
there exist positive constants $c_4$ and $c_5$ such that 
for every $\epsilon \in (0, \sigma)$  
and every $n \ge 2$, 
\[ P\left( \left| \frac{1}{n} \sum_{i=1}^{n} \textup{Re}(f(X_i)) - \textup{Re}(f(\theta)) \right| > \frac{c_3 \epsilon}{2} \right) \le c_4 \exp(-c_5 n \epsilon^2), \] 
and 
\[ P\left( \left| \frac{1}{n} \sum_{i=1}^{n} \textup{Im}(f(X_i)) - \textup{Im}(f(\theta)) \right| > \frac{c_3 \epsilon}{2} \right) \le c_4 \exp(- c_5 n \epsilon^2).  \]

Now we have the assertion. 
\end{proof}

\subsection{Proofs of assertions in Section 3}

Let $H(t) := E\left[h(X, t)\right]$ for $X$ following $C(\theta)$.  
Then, by the Cauchy integral formula, 
\[ H(t) = \frac{\theta - t}{\theta - \overline{t}} = h(\theta, t).  \]
It is continuous on $\overline{H}$. 
$H(t) = 0$ if and only if $t = \theta$. 

\begin{Lem}\label{LDP-cpt}
There exists a compact set $K$ of $\mathbb H$ and a constant $\alpha \in (0,1)$ such that $\theta \in K$ and 
\[ P(\hat{\theta}_n \notin K) = O(\alpha^n).   \]
\end{Lem}

\begin{proof}
Let $F_{\theta}(z) := h(z, \theta)$ 
and 
$$K_m := F_{\theta}^{-1}\left(\left\{z \in \mathbb{C}  : |z|^2 \le \frac{m}{m+1} \right\}\right), \ m \ge 1.$$
We see that each $K_m$ is a compact subset of $\mathbb H$ and $\theta \in K_1$. 
By the definition of $H$, we see that 
\[ \inf_{t \notin K_1} |H(t)| > 0.  \]
Let 
\[ R_m^{(k)}(x_1, \dots, x_k) := \sup_{t \notin K_m} \frac{1}{k} \left| \sum_{j=1}^{k} \left(\frac{h(x_j, t)}{H(t)} - 1\right) \right|, \ \ x_1, \dots, x_k \in \mathbb{R}.  \]
This is decreasing with respect to $m$. 
Let 
\[ R^{(k)}(x_1, \dots, x_k) := \lim_{m \to \infty} R_m^{(k)}(x_1, \dots, x_k). \]
Since $|h(x_j, t)| = 1$,
we see that 
\[ |R_m^{(k)}(x_1, \dots, x_k)| \le 1 + \frac{1}{\inf_{t \notin K_1} |H(t)|}. \]
Hence, 
\[ |R^{(k)}(x_1, \dots, x_k)| \le 1 + \frac{1}{\inf_{t \notin K_1} |H(t)|}. \]

Since 
\[ \frac{h(x_j, t)}{H(t)} - 1 = \frac{2 \textup{Im}(t) (x_j -\theta)}{(x_j - \overline{t})(\theta - t)}, \]
we see that if $a \in \mathbb{R} \setminus \{x_j\}$, then, 
\[ \lim_{t \to a} \frac{h(x_j, t)}{H(t)} - 1 = 0.  \]
and furthermore, 
\[ \lim_{|t| \to \infty; t \in \mathbb H} \frac{h(x_j, t)}{H(t)} - 1 = 0. \]

Therefore, 
\[ \lim_{t \to a} \sum_{j=1}^{k} \left(\frac{h(x_j, t)}{H(t)} - 1\right) = 0, \ a \in \mathbb{R} \setminus\{x_1, \dots, x_k\}.  \]
and, 
\[ \lim_{t \to \infty} \sum_{j=1}^{k} \left(\frac{h(x_j, t)}{H(t)} - 1\right) = 0. \]
 Furthermore, if $x_1, \dots, x_k$ are distinctive, then, 
\[ \limsup_{t \to a} \left|  \sum_{j=1}^{k} \left(\frac{h(x_j, t)}{H(t)} - 1\right) \right|  \le 1 + \frac{1}{\inf_{t \notin K_1} |H(t)|}, \ a \in \{x_1, \dots, x_k\}.  \]
Therefore we see that if $x_1, \dots, x_k$ are distinctive, then, 
\[ |R^{(k)}(x_1, \dots, x_k)| \le \frac{1}{k} \left(1 + \frac{1}{\inf_{t \notin K_1} |H(t)|}\right).  \]

Let $k_0$ be an integer such that 
\[ k_0 > 2 + \frac{1}{\inf_{t \notin K_1} |H(t)|}. \]

Now by following the argument in the proof of \cite[Theorem 3.8]{Arcones2006}, we see that for every $\epsilon \in (0,1)$, 
\[ \left\{ \sum_{j=1}^{n} R_m^{(k_0)}\left(X_{(j-1)k_0 +1}, \cdots, X_{j k_0} \right)  \le n(1-\epsilon) \right\} \subset \left\{ \hat{\theta}_{k_0 n} \in K_m \right\}. \]



By this and the fact that $\{(X_{(j-1)k_0 +1}, \cdots, X_{j k_0})\}_j$ are independent, 
\[ P( \hat{\theta}_{k_0 n} \notin K_m) \le P\left( \sum_{j=1}^{n} R_m^{(k_0)}\left(X_{(j-1)k_0 +1}, \cdots, X_{j k_0} \right) > n(1-\epsilon)  \right)\]
\[ \le \exp(-n(1-\epsilon)) E\left[ \exp\left( \sum_{j=1}^{n} R_m^{(k_0)}\left(X_{(j-1)k_0 +1}, \cdots, X_{j k_0} \right) \right)\right] \]
\[ = \left(\exp(- (1-\epsilon)) E\left[\exp\left(R_m^{(k_0)}\left(X_{1}, \cdots, X_{k_0} \right)\right)\right]  \right)^n.  \]

By the bounded convergence theorem, 
\[ \lim_{m \to \infty} E\left[\exp\left(R_m^{(k_0)}\left(X_{1}, \cdots, X_{k_0} \right)\right)\right] = E\left[\exp\left(R^{(k_0)}\left(X_{1}, \cdots, X_{k_0} \right)\right)\right]. \]

Since $X_{1}, \cdots, X_{k_0} $ are distinctive almost surely, 
\[ E\left[\exp\left(R^{(k_0)}\left(X_{1}, \cdots, X_{k_0} \right)\right)\right] \le \exp\left( \frac{1}{k_0} \left(1 + \frac{1}{\inf_{t \notin K_1} |H(t)|}\right) \right). \]

By recalling the definition of $k_0$, if we take sufficiently small $\epsilon$ and sufficiently large $m$, 
\[ \exp(- (1-\epsilon)) E\left[\exp\left(R_m^{(k_0)}\left(X_{1}, \cdots, X_{k_0} \right)\right)\right] < 1.\] 

Now we have the assertion for the case that $n$ is a multiple of $k_0$. 
Even if  $n$ is not a multiple of $k_0$, 
with remark the fact that 
$\dfrac{h(x, t)}{H(t)}$ is uniformly bounded, 
we see that for a positive constant $C$, 
\[ P( \hat{\theta}_{n} \notin K_m) \le C\left(\exp(- (1-\epsilon)) E\left[\exp\left(R_m^{(k_0)}\left(X_{1}, \cdots, X_{k_0} \right)\right)\right]  \right)^{\lfloor n/k_0\rfloor},  \]
where $\lfloor n/k_0\rfloor$ denotes the integer part of $n/k_0$. 
\end{proof}

\begin{Rem}
Although we have used the techniques in the proof of  \cite[Theorem 3.8]{Arcones2006}, 
Condition (v) in \cite[Theorem 3.8]{Arcones2006} does not hold in our case. 
Our main idea is introducing $R_m^{(k)}\left(X_{1}, \cdots, X_{k} \right)$ for general $k$. 
\end{Rem}

Let $h_1 (x,t) := \textup{Re}(h(x,t))$ and $h_2 (x,t) := \textup{Im}(h(x,t))$. 
As a function of $t \in K$, 
$ \frac{1}{n} \sum_{j=1}^{n} h_i (X_j, t)$ is a $\ell_{\infty}(K)$-valued random variable. 

Let 
\[ \mathcal{L} := \left\{ f : \mathbb{R} \to \mathbb{R}\  \middle| \ \exists \lambda < +\infty \textup{ such that } E\left[\exp(\lambda |f(X_1)|)\right] < +\infty  \right\}. \]
Let $\mathcal{L}^{*}$ be the dual space of $\mathcal{L}$. 
Let 
\[ J(l) := \sup_{f \in \mathcal{L}} \left(l(f) - \log E[\exp(f(X_1))] \right), \ \ l \in \mathcal{L}^{*}.  \]
Let $\ell_{\infty}(K)$ be the set of bounded continuous functions on $K$. 
Let $P_{*}$ and $P^{*}$ be the inner and outer measures of a probability measure $P$, respectively.  

\begin{Lem}\label{LDP-h}
Assume that $i = 1$ or $2$. 
For each non-empty compact subset $K$ of $\mathbb{H}$, 
$\left\{ \frac{1}{n} \sum_{j=1}^{n} h_i (X_j, t) \right\}_n$ follows the large deviation principle in $\ell_{\infty}(K)$. 
Specifically, there exists a good rate function $I$ on $\ell_{\infty}(K)$ such that (i) $\{I \le c\}$ is compact in $\ell_{\infty}(K)$ for every $c > 0$, \\
(ii) for every open subset $U$ of  $\ell_{\infty}(K)$,  
\[ \liminf_{n \to \infty} \frac{1}{n} \log P_{*}\left(\frac{1}{n} \sum_{j=1}^{n} h_i (X_j, t)  \in U\right)  \ge -\inf_{z \in U} I(z), \]
and, (iii)  for every closed subset $F$ of  $\ell_{\infty}(K)$,  
\[ \liminf_{n \to \infty} \frac{1}{n} \log P^{*}\left(\frac{1}{n} \sum_{j=1}^{n} h_i (X_j, t)  \in F \right)  \le -\inf_{z \in F} I(z). \]
Furthermore, 
\[ I_i(z) = \inf\left\{J(l) : l \in \mathcal{L}^{*}, \  l\left(h_i(\cdot, t)\right) = z(t) \textup{ for every } t \in K \right\}, \ z \in \ell_{\infty}(K). \]
\end{Lem}

\begin{proof}
It suffices to check the conditions in \cite[Theorem 2.5]{Arcones2006}.  
We remark that for every $x \in \mathbb{R}$ and every $t \in \mathbb{H}$, 
\[ |h_i (x,t)| \le |h(x,t)| = 1. \]
This implies \cite[Theorem 2.5 conditions (i) and (ii)]{Arcones2006}. 

Let $s, t \in \mathbb{H}$ and $|s - t| \le |t|/4$. 
Since 
\begin{equation}\label{h-difference} 
h(x,s) - h(x,t) = \frac{x-s}{(x-\overline{s})(x - \overline{t})} (\overline{s} - \overline{t}) + \frac{t-s}{x - \overline{t}}, \ x \in \mathbb{R},  
\end{equation} 
we see that 
\begin{equation}\label{after-h-difference}   
\left|   h(x,s) - h(x,t)  \right| \le 2 \frac{|s-t|}{\textup{Im}(t)}.  
\end{equation} 
Hence \cite[Theorem 2.5 condition (iii)]{Arcones2006} holds.   
\end{proof}

\begin{Prop}\label{LD-consistent}
For every $\epsilon > 0$, 
there exists a constant $\beta \in (0,1)$ such that 
\[ P\left( \left|\hat{\theta}_n - \theta\right| \ge \epsilon\right) = O(\beta^n).  \]
\end{Prop}

\begin{proof}
By Lemma \ref{LDP-cpt}, there exists a compact set $K$ of $\mathbb{H}$ and $\alpha \in (0,1)$ such that 
\[  P\left(\hat{\theta}_n \notin K  \right) = O(\alpha^n). \]

Let 
\[ F := \left\{t \in \mathbb{H} : |t - \theta| \ge \epsilon \right\}.  \]
Then, 
\[  P\left( \left|\hat{\theta}_n - \theta\right| \ge \epsilon\right) \le P\left(\hat{\theta}_n \notin K  \right) + P\left(\hat{\theta}_n \in K \cap F  \right). \] 
Hence it suffices to show that there exists a constant $\beta \in (0,1)$ such that 
\[ P\left(\hat{\theta}_n \in K \cap F  \right) =  O(\beta^n). \] 

Let $H_1 (t)$ and $H_2 (t)$ be the real and imaginary parts of $H(t)$, respectively.  
Since 
\[ \epsilon_0 := \inf_{t \in F} |H(t)| > 0,  \]
we see that 
\[ \left\{t \in F : \left|H_1 (t)\right| \le \epsilon_0 /4  \right\} \subset \left\{t \in F : \left|H_2 (t)\right| \ge \epsilon_0 /4  \right\}.  \]
Let 
\[ F_1 := \left\{t \in F : \left|H_1 (t)\right| \ge \epsilon_0 /4  \right\},  \textup{ and }   F_2  := \left\{t \in F : \left|H_1 (t)\right| \le \epsilon_0 /4  \right\}.  \]
Then $F = F_1 \cup F_2$. 
Therefore, 
\[ P\left(\hat{\theta}_n \in K \cap F  \right) \le P^{*}\left(\textup{there exists $t \in F \cap K$ such that } \frac{1}{n} \sum_{j=1}^{n} h (X_j, t) = 0 \right) \]
\[ \le P^{*}\left(\textup{there exists $t \in F_1 \cap K$ such that } \frac{1}{n} \sum_{j=1}^{n} h_1 (X_j, t) = 0 \right) \] 
\[+ P^{*}\left(\textup{there exists $t \in F_2 \cap K$ such that } \frac{1}{n} \sum_{j=1}^{n} h_2 (X_j, t) = 0 \right). \]

Let   
\[ C_i := \left\{z \in \ell_{\infty}(F_i \cap K) : \inf_{t \in F_i \cap K} \left|z(t) \right| = 0  \right\}, \ i = 1,2. \]
These are closed subsets of $\ell_{\infty} (F_i \cap K)$, $i = 1,2$.    
We see that 
\[ P^{*}\left(\textup{there exists $t \in F_i \cap K$ such that } \frac{1}{n} \sum_{j=1}^{n} h_i (X_j, t) = 0 \right) = P^* \left( \frac{1}{n} \sum_{j=1}^{n} h_i (X_j, t) \in C_i \right).  \]

Then, by Lemma \ref{LDP-h}, we see that  for $i = 1,2$,  
\[ \limsup_{n \to \infty} \frac{1}{n} \log P^{*}\left( \frac{1}{n} \sum_{j=1}^{n} h_i (X_j, t) \in C_i \right) \]
\[ \le - \inf_{z \in  C_i}   \inf\left\{J(l) : l \in \mathcal{L}^{*}, \  l\left(h_i(\cdot, t)\right) = z(t) \textup{ for every } t \in F_i \cap K \right\}  \]
\begin{equation}\label{LDP-J1} 
= -\inf \left\{J(l) : l \in \mathcal{L}^{*}, \ \inf_{t  \in F_i \cap K} \left| l\left(h_i(\cdot, t)\right) \right|  = 0 \right\}.  
\end{equation}

By \cite[Eq. (31)]{Arcones2006}, we see that 
\[  -\inf \left\{J(l) : l \in \mathcal{L}^{*}, \ \inf_{t  \in F_i \cap K} \left| l\left(h_i(\cdot, t)\right) \right|  = 0 \right\} =  -\inf_{t  \in F_i \cap K}  \inf \left\{J(l) : l \in \mathcal{L}^{*}, \   l\left(h_i(\cdot, t)\right)  = 0 \right\} \]
\begin{equation}\label{LDP-I1} 
= \sup_{t  \in F_i \cap K} \inf_{\lambda \in \mathbb{R}}  \log E\left[\exp\left( \lambda h_i (X_1, t) \right)\right].  
\end{equation}

Since $h_i$ is bounded,  it holds that 
\[ \frac{d}{d\lambda} \bigg|_{\lambda = 0} \log E\left[\exp\left( \lambda h_i (X_1, t) \right)\right] = H_i (t) \ne 0, \ t \in F_i \cap K.   \]

By this and the fact that $\log E\left[\exp\left( \lambda h_i (X_1, t) \right)\right] = 0$ if $\lambda = 0$, 
we see that 
\[ \inf_{\lambda \in \mathbb{R}} \log E\left[\exp\left( \lambda h_i (X_1, t) \right)\right] < 0, \ \ t \in F_i \cap K.  \]

Let 
\[ I_i (t) := - \inf_{\lambda \in \mathbb{R}} \log E\left[\exp\left( \lambda h_i (X_1, t) \right)\right], \ \ t \in F_i \cap K. \]
Then, $I_i (t) > 0, \ \ t \in F_i \cap K$ and this function is lower-semicontinuous. 
Hence, 
\[ \inf_{t \in F_i \cap K} I_i (t) > 0.  \]
By this, \eqref{LDP-J1},  and \eqref{LDP-I1}, 
\[ \limsup_{n \to \infty} \frac{1}{n} \log P^{*}\left( \frac{1}{n} \sum_{j=1}^{n} h_i (X_j, t) \in C_i \right)  \le - \inf_{t \in F_i \cap K} I_i (t) < 0, \ i = 1,2. \]
Thus we have the assertion. 
\end{proof}

\begin{proof}[Proof of Theorem \ref{Bahadur-MLE}]
We adopt \cite[Theorem 3]{Shen2001}.  
The condition being LD-consistent in the statement of \cite[Theorem 3]{Shen2001} is equivalent with Proposition \ref{LD-consistent}.  
Let $\psi(x, \theta) := \log p(x; (\mu, \sigma))$. 
Then, we can show \cite[(C1) and (C2) in Theorem 3]{Shen2001} by some calculations for the first and second orders of the partial derivatives of $\psi$ with respect to $\mu$ and $\sigma$. 
Now we can apply \cite[Theorem 3]{Shen2001}, and then, the assertion follows from this and \cite[Proposition 2]{Shen2001}.   
\end{proof}

\subsection{Proofs of assertions in Section 4}

\begin{proof}[Proof of Theorem \ref{main}]
We see that for every $C_1 \in (0,1)$ and $\epsilon > 0$, 
\[ P\left( \left|Z_n - \theta\right| > \epsilon\right) \le P\left(|Z_n - \hat{\theta}_n| > c_1 \epsilon \right) + P\left(|\hat{\theta}_n - \theta| >  (1- c_1) \epsilon \right). \]

By \eqref{b-formula}, 
\[ \lim_{c_1  \to +0} b((1-c_1)\epsilon, \theta) = b(\epsilon, \theta). \]

By this and Theorem \ref{Bahadur-MLE},   
we see that for every $\eta > 0$, there exists $\epsilon_0 > 0$ such that  for every  $\epsilon \in (0, \epsilon_0)$ and $c_1 \in (0,\epsilon_0)$,   
\[ P\left( \left|\hat{\theta}_n - \theta\right| >  (1- c_1) \epsilon \right) = O\left(\exp\left(-n (b(\epsilon, \theta) - \eta) \right) \right), \]
where the large order depends on $c_1$, $\eta$ and $\epsilon$.   

By this and \eqref{b-order}, it suffices to show that for sufficiently small  $\epsilon > 0$ 
\begin{equation}\label{WTS-main} 
P\left( \left|Z_n - \hat{\theta}_n \right| >  \epsilon \right) = O\left( \exp(- n \epsilon^{3/2}) \right).  
\end{equation} 

By using the fact that $\hat{\theta}_n$ is MLE and \eqref{h-difference}, we see that 
\[ Z_n - \hat{\theta}_n = Y_n - \hat{\theta}_n  + \frac{2 \textup{Im}(Y_n) i}{n} \sum_{j =1}^{n} \left(\frac{X_j - Y_n}{X_j - \overline{Y_n}}   
-  \frac{X_j - \hat{\theta}_n}{X_j - \overline{\hat{\theta}_n}}  \right) \]
\[ =  \frac{Y_n -  \hat{\theta}_n}{n} \sum_{j=1}^{n} \frac{X_j - Y_n}{X_j - \overline{Y_n}}   + \frac{\overline{Y_n -  \hat{\theta}_n}}{n} \sum_{j=1}^{n} \frac{X_j - \hat{\theta}_n}{X_j - \overline{\hat{\theta}_n}}  -  \frac{\overline{Y_n -  \hat{\theta}_n}}{n} \sum_{j=1}^{n} \frac{X_j - \hat{\theta}_n}{X_j - \overline{\hat{\theta}_n}} \frac{X_j - Y_n}{X_j - \overline{Y_n}}.      \]

Hence, 
\[ P\left( \left|Z_n - \hat{\theta}_n \right| >  \epsilon \right)  \le P\left( |Y_n -  \hat{\theta}_n| \cdot  \left|\frac{1}{n} \sum_{j=1}^{n} \frac{X_j - Y_n}{X_j - \overline{Y_n}}\right| > \frac{\epsilon}{3} \right) \]
\[ +  P\left( |Y_n -  \hat{\theta}_n| \cdot \left|\frac{1}{n} \sum_{j=1}^{n} \frac{X_j - \hat{\theta}_n}{X_j - \overline{\hat{\theta}_n}}\right| > \frac{\epsilon}{3} \right)  +  P\left( |Y_n -  \hat{\theta}_n| \cdot  \left|\frac{1}{n} \sum_{j=1}^{n} \frac{X_j - \hat{\theta}_n}{X_j - \overline{\hat{\theta}_n}}     \frac{X_j - Y_n}{X_j - \overline{Y_n}}             \right| > \frac{\epsilon}{3} \right). \] 

By Theorems \ref{LDP-chernoff} and \ref{Bahadur-MLE}, 
we see that there exists a constant $c_2 > 0$ such that for sufficiently small $\eta > 0$, 
\begin{equation}\label{tail-Y}  
P\left( \left|Y_n - \theta \right| > \eta \right) = O(\exp(-c_2 n \eta^2)) 
\end{equation}
and 
\begin{equation}\label{tail-MLE}  
P\left( \left|\hat{\theta}_n - \theta  \right| > \eta \right) = O(\exp(-c_2 n \eta^2)).  
\end{equation} 
Hence,  there exists a constant $c_3 > 0$ such that for sufficiently small $\eta > 0$, 
\begin{equation}\label{tail-Y-MLE}   
P\left( \left|Y_n - \hat{\theta}_n \right| > \eta \right) = O(\exp(-c_3 n \eta^2)).  
\end{equation}

By the Cram\'er-Chernoff method, 
we also obtain that  there exists a constant $c_4 > 0$ such that for sufficiently small $\eta > 0$, 
\begin{equation}\label{tail-Cayley-1} 
P\left(   \left|\frac{1}{n} \sum_{j=1}^{n} \frac{X_j - \theta}{X_j - \overline{\theta}}  \right| > \eta \right) = O(\exp(-c_4 n \eta^2)).  
\end{equation}
and 
\begin{equation}\label{tail-Cayley-2} 
P\left(   \left|\frac{1}{n} \sum_{j=1}^{n} \left(\frac{X_j - \theta}{X_j - \overline{\theta}}\right)^2   \right| > \eta \right) = O(\exp(-c_4 n \eta^2)).  
\end{equation}
In the above we have used that 
\[ E \left[ \frac{X_j - \theta}{X_j - \overline{\theta}} \right] =  E \left[ \left(\frac{X_j - \theta}{X_j - \overline{\theta}}\right)^2 \right] = 0, \]
both of which follow from the Cauchy integral formula. 

By \eqref{after-h-difference}, 
we see that 
\[ \left|\frac{1}{n} \sum_{j=1}^{n} \frac{X_j - Y_n}{X_j - \overline{Y_n}}\right| \le \left|\frac{1}{n} \sum_{j=1}^{n} \frac{X_j - \theta}{X_j - \overline{\theta}}\right| + 2 \frac{|Y_n - \theta|}{\textup{Im}(\theta)}  \]
and
\[ \left|\frac{1}{n} \sum_{j=1}^{n} \frac{X_j - \hat{\theta}_n}{X_j - \overline{\hat{\theta}_n}}\right| \le \left|\frac{1}{n} \sum_{j=1}^{n} \frac{X_j - \theta}{X_j - \overline{\theta}}\right| + 2 \frac{\left|\hat{\theta}_n - \theta\right|}{\textup{Im}(\theta)}. \]

By these estimates, \eqref{tail-Y-MLE},  \eqref{tail-Cayley-1}, and \eqref{tail-Y}, 
 we see that there exists a constant $c_5 > 0$ such that for sufficiently small $\epsilon > 0$, 
\[ P\left( \left|Y_n -  \hat{\theta}_n\right| \cdot  \left|\frac{1}{n} \sum_{j=1}^{n} \frac{X_j - Y_n}{X_j - \overline{Y_n}}\right| > \frac{\epsilon}{3} \right) \]
\[\le 2 P\left( \left|Y_n -  \hat{\theta}_n\right|  > \sqrt{\frac{\epsilon}{6}} \right) + P\left(\left|\frac{1}{n} \sum_{j=1}^{n} \frac{X_j - \theta}{X_j - \overline{\theta}}\right| > \sqrt{\frac{\epsilon}{6}} \right) + P\left(2 \frac{|Y_n - \theta|}{\textup{Im}(\theta)}   > \sqrt{\frac{\epsilon}{6}} \right) \]
\[ = O\left(\exp(-c_5 n \epsilon) \right). \]

In the same manner, 
by noting \eqref{tail-Y-MLE},  \eqref{tail-Cayley-1}, and \eqref{tail-MLE}, 
we see that  there exists a constant $c_6 > 0$ such that for sufficiently small $\epsilon > 0$, 
\[  P\left( \left|Y_n -  \hat{\theta}_n\right| \cdot  \left|\frac{1}{n} \sum_{j=1}^{n} \frac{X_j -  \hat{\theta}_n}{X_j - \overline{ \hat{\theta}_n}}\right| > \frac{\epsilon}{3} \right) = O\left(\exp(-c_6 n \epsilon) \right). \]

We remark that 
\[  \frac{X_j - \hat{\theta}_n}{X_j - \overline{\hat{\theta}_n}}     \frac{X_j - Y_n}{X_j - \overline{Y_n}} -  \left(\frac{X_j - \theta}{X_j - \overline{\theta}}\right)^2 \]
\[ = \left( \frac{X_j - \hat{\theta}_n}{X_j - \overline{\hat{\theta}_n}}  -  \frac{X_j - \theta}{X_j - \overline{\theta}}\right)   \frac{X_j - Y_n}{X_j - \overline{Y_n}} +   \left(\frac{X_j - Y_n}{X_j - \overline{Y_n}} - \frac{X_j - \theta}{X_j - \overline{\theta}}\right) \frac{X_j - \theta}{X_j - \overline{\theta}}. \]

By this, \eqref{after-h-difference} and the fact that $|X_j - \theta| = |X_j - \overline{\theta}|$, 
we see that 
\[ \left|\frac{1}{n} \sum_{j=1}^{n} \frac{X_j - \hat{\theta}_n}{X_j - \overline{\hat{\theta}_n}}     \frac{X_j - Y_n}{X_j - \overline{Y_n}}             \right|\]
\[  \le  \left|\frac{1}{n} \sum_{j=1}^{n} \left(\frac{X_j - \theta}{X_j - \overline{\theta}}\right)^2   \right|  + 2 \frac{|Y_n - \theta|}{\textup{Im}(\theta)}   + 2 \frac{|  \hat{\theta}_n - \theta|}{\textup{Im}(\theta)} + 4 \frac{|Y_n - \theta| \cdot |  \hat{\theta}_n - \theta|}{\textup{Im}(\theta)^2}.  \]



By these estimates, we see that for every $\epsilon > 0$, 
\[ P\left( |Y_n -  \hat{\theta}_n| \cdot  \left|\frac{1}{n} \sum_{j=1}^{n} \frac{X_j - \hat{\theta}_n}{X_j - \overline{\hat{\theta}_n}}     \frac{X_j - Y_n}{X_j - \overline{Y_n}}             \right| > \frac{\epsilon}{3} \right) \]
\[\le 4 P\left( \left|Y_n -  \hat{\theta}_n\right|  > \sqrt{\frac{\epsilon}{12}} \right) + P\left(\left|\frac{1}{n} \sum_{j=1}^{n}  \left(\frac{X_j - \theta}{X_j - \overline{\theta}} \right)^2 \right| > \sqrt{\frac{\epsilon}{12}} \right) \]
\[+ P\left(2 \frac{|Y_n - \theta|}{\textup{Im}(\theta)}   > \sqrt{\frac{\epsilon}{12}} \right) + P\left(2 \frac{|\hat{\theta}_n - \theta|}{\textup{Im}(\theta)}   > \sqrt{\frac{\epsilon}{12}} \right) + P\left(4 \frac{|Y_n - \theta| \cdot |  \hat{\theta}_n - \theta|}{\textup{Im}(\theta)^2}  > \sqrt{\frac{\epsilon}{12}} \right). \]

We remark that  for sufficiently small $\epsilon > 0$, 
\[ P\left(4 \frac{|Y_n - \theta| \cdot |  \hat{\theta}_n - \theta|}{\textup{Im}(\theta)^2}  > \sqrt{\frac{\epsilon}{12}} \right) \]
\[ \le P\left(2 \frac{|Y_n - \theta|}{\textup{Im}(\theta)}   > 1  \right) +  P\left(4 \frac{|Y_n - \theta| \cdot |  \hat{\theta}_n - \theta|}{\textup{Im}(\theta)^2}  > \sqrt{\frac{\epsilon}{12}}, \ 2 \frac{|Y_n - \theta|}{\textup{Im}(\theta)} \le 1  \right)  \]
\[ \le P\left(2 \frac{|Y_n - \theta|}{\textup{Im}(\theta)}   > \sqrt{\frac{\epsilon}{12}} \right) + P\left(2 \frac{|\hat{\theta}_n - \theta|}{\textup{Im}(\theta)}   > \sqrt{\frac{\epsilon}{12}} \right).  \] 

By these estimates and  \eqref{tail-Cayley-2}, \eqref{tail-Y}, \eqref{tail-MLE} and  \eqref{tail-Y-MLE}, 
we see that  there exists a constant $c_7  > 0$ such that for sufficiently small $\epsilon > 0$, 
\[ P\left( |Y_n -  \hat{\theta}_n| \cdot  \left|\frac{1}{n} \sum_{j=1}^{n} \frac{X_j - \hat{\theta}_n}{X_j - \overline{\hat{\theta}_n}}     \frac{X_j - Y_n}{X_j - \overline{Y_n}}             \right| > \frac{\epsilon}{3} \right) = O\left(\exp(-c_7 n \epsilon) \right).\]

Thus we have \eqref{WTS-main} and this completes the proof. 
\end{proof} 



\begin{proof}[Proof of Theorem \ref{main-var}]
We see that 
\[ Z_n - \theta = Y_n - \theta + \frac{2 \textup{Im}(Y_n) i}{n} \sum_{j=1}^{n} \left(\frac{X_j - Y_n}{X_j - \overline{Y_n}} - \frac{X_j - \theta}{X_j - \overline{\theta}} \right) 
+ \frac{2 \textup{Im}(Y_n) i}{n} \sum_{j=1}^{n}\frac{X_j - \theta}{X_j - \overline{\theta}}. \]

By \eqref{after-h-difference}, 
\[ \left|  Y_n - \theta + \frac{2 \textup{Im}(Y_n) i}{n} \sum_{j=1}^{n} \left(\frac{X_j - Y_n}{X_j - \overline{Y_n}} - \frac{X_j - \theta}{X_j - \overline{\theta}} \right) \right| \le 9 \left|  Y_n - \theta \right|.  \]

We see that 
\[ \left| \frac{2 \textup{Im}(Y_n) i}{n} \sum_{j=1}^{n}\frac{X_j - \theta}{X_j - \overline{\theta}} \right| \le 2  \left|  Y_n - \theta \right|  +  2 \textup{Im}(\theta) \left| \frac{1}{n} \sum_{j=1}^{n}\frac{X_j - \theta}{X_j - \overline{\theta}} \right|. \]

Hence, 
\[ \sqrt{n}|Z_n - \theta| \le 11 \sqrt{n} |Y_n - \theta| + \frac{2\textup{Im}(\theta)}{\sqrt{n}} \left|  \sum_{j=1}^{n}\frac{X_j - \theta}{X_j - \overline{\theta}} \right|.  \]

Now we should recall that the generators of quasi-arithmetic means are restricted to the cases that 
$f(x) = \log (x+\gamma), \ \gamma \in \overline{\mathbb{H}}$ and $f(x) = \dfrac{x - \gamma}{x - \overline{\gamma}}, \ \gamma \in \mathbb{H}$. 
By \cite[Theorem 4.2(iii) and Theorem 4.4(iii)]{Akaoka2021-1} and \eqref{main-var-ass}, 
we can apply Billingsley \cite[Theorem 3.6]{Billingsley1999} 
and have  that
$ \left\{ \left| \sqrt{n} (Y_n - \theta) \right|^2 \right\}_n$ is uniformly integrable. 
By the central limit theorem, 
\[ \frac{1}{\sqrt{n}}  \sum_{j=1}^{n}\frac{X_j - \theta}{X_j - \overline{\theta}} \Rightarrow N\left(0, \frac{1}{2} I_2 \right),  \  n \to \infty. \]
By the Cauchy integral formula, 
we also see that 
$E\left[ \left|  \frac{1}{\sqrt{n}}  \sum_{j=1}^{n}\frac{X_j - \theta}{X_j - \overline{\theta}} \right| \right] = 1$. 
Therefore we can apply Billingsley \cite[Theorem 3.6]{Billingsley1999} again 
and have  that
$ \left\{ \left|  \frac{1}{\sqrt{n}}  \sum_{j=1}^{n}\frac{X_j - \theta}{X_j - \overline{\theta}} \right| ^2 \right\}_n$ is uniformly integrable. 
Therefore, $ \left\{ \left| \sqrt{n}(Z_n - \theta)  \right|^2 \right\}_n$ is also uniformly integrable
By this and \eqref{Z-Fisher}, we have the assertion. 
\end{proof}

\subsection{Proofs of assertions in Section 5}

\begin{proof}[Proof of Theorem \ref{Bahadur-cc}]
We first remark that for sufficiently small $\epsilon > 0$, 
\[ \widetilde b(\epsilon, w) = \log \left(1+ \frac{\epsilon^2}{(1 - |w|^2)(1 - (|w| + \epsilon)^2)} \right). \]
Hence, 
\begin{equation}\label{tilde-b-order} 
\lim_{\epsilon \to +0} \frac{\widetilde b(\epsilon, w)}{\epsilon^2} = \frac{1}{(1 - |w|^2)^2}.  
\end{equation}

Let $\theta := \phi_{\alpha}^{-1}(w)$.  
By the complex mean-value  theorem \cite[Theorem 2.2]{Evard1992},  
\[ P(|W_n - w| > \epsilon) \le P\left(|Z_n - \theta| > \frac{\epsilon}{\sup_{z; |z-\theta| \le \epsilon} |\phi_{\alpha}^{\prime}(z)|}\right).\]
Let 
\[ \epsilon^{\prime} :=  \frac{\epsilon}{\sup_{z; |z-\theta| \le \epsilon} \left|\phi_{\alpha}^{\prime}(z)\right|}. \]
Then, by \eqref{tilde-b-order} and \eqref{b-order}, 
\[ \lim_{\epsilon \to +0} \frac{\widetilde b(\epsilon, w)}{b(\epsilon^{\prime}, \theta)} = \lim_{\epsilon \to +0} \frac{\widetilde b(\epsilon, w)}{\epsilon^2} \frac{(\epsilon^{\prime})^2}{b(\epsilon^{\prime}, \theta)} \frac{\epsilon^2}{(\epsilon^{\prime})^2} = \frac{4 \textup{Im}(\theta)^2  \left|\phi_{\alpha}^{\prime}(\theta)\right|^2}{(1 - |w|^2)^2} = 1. \]

Now the assertion follows from this and Theorem \ref{main}. 
\end{proof}

\begin{proof}[Proof of Theorem \ref{Var-CR-cc}]
Let $\theta := \phi_{\alpha}^{-1}(w)$. 
We see that 
\[ n E\left[ \left|\phi_{\alpha}(Z_n) - \phi_{\alpha}(\theta) - \phi_{\alpha}^{\prime}(\theta) (Z_n - \theta)\right|^2\right] = E\left[ n |Z_n - \theta|^2 \left| \frac{\phi_{\alpha}(Z_n) - \phi_{\alpha}(\theta)}{Z_n - \theta} -  \phi_{\alpha}^{\prime}(\theta)  \right|^2 \right]. \]
We also have that $\sup_{z \in \mathbb{H}} |\phi_{\alpha}^{\prime}(z)| \le 2 \textup{Im}(\alpha)$ 
and 
$\sup_{z \in \mathbb{H}} |\phi_{\alpha}^{\prime\prime}(z)| \le 4 \textup{Im}(\alpha)$.

By using \cite[Theorem 2.2]{Evard1992},  
\[ \left| \frac{\phi_{\alpha}(Z_n) - \phi_{\alpha}(\theta)}{Z_n - \theta} -  \phi_{\alpha}^{\prime}(\theta)  \right| \le 4 \textup{Im}(\alpha)\min\left\{1, |Z_n - \theta| \right\}. \]

Let $\epsilon > 0$. 
Then, 
\[ E\left[ n |Z_n - \theta|^2 \left| \frac{\phi_{\alpha}(Z_n) - \phi_{\alpha}(\theta)}{Z_n - \theta} -  \phi_{\alpha}^{\prime}(\theta)  \right|^2 \right] \]
\[\le 4 \textup{Im}(\alpha) E\left[ n |Z_n - \theta|^2, \ |Z_n - \theta| > \epsilon \right] + 4\textup{Im}(\alpha) \epsilon^2 E\left[ n |Z_n - \theta|^2, \ |Z_n - \theta| \le \epsilon \right]. \]

By the final part of the proof of Theorem \ref{main-var},  
$\lim_{n \to \infty} P(|Z_n - \theta| > \epsilon) = 0$ 
and 
$\left\{n |Z_n - \theta|^2 \right\}_n$ is uniformly integrable.  
Hence, 
\[ \lim_{n \to \infty} E\left[ n |Z_n - \theta|^2, \ |Z_n - \theta| > \epsilon \right] = 0.  \]
Hence, 
\[ \limsup_{n \to \infty} E\left[ n |Z_n - \theta|^2 \left| \frac{\phi_{\alpha}(Z_n) - \phi_{\alpha}(\theta)}{Z_n - \theta} -  \phi_{\alpha}^{\prime}(\theta)  \right|^2 \right] \le 4 \textup{Im}(\alpha) \epsilon^2 \sup_{n} E\left[ n |Z_n - \theta|^2\right].  \]
Since we can take $\epsilon > 0$ arbitrarily small, we see that 
\[ \lim_{n \to \infty} n E\left[ \left|\phi_{\alpha}(Z_n) - \phi_{\alpha}(\theta) - \phi_{\alpha}^{\prime}(\theta) (Z_n - \theta)\right|^2\right] = 0. \] 

Hence, 
\[ \lim_{n \to \infty} n E\left[|\phi_{\alpha}(Z_n) - \phi_{\alpha}(\theta)|^2\right] = 4  \textup{Im}(\theta)^2 |\phi_{\alpha}^{\prime}(\theta)|^2.  \]

The assertion now follows from this and 
\[ 4  \textup{Im}(\phi_{\alpha}^{-1}(w))^2 |\phi_{\alpha}^{\prime}(\phi_{\alpha}^{-1}(w))|^2 = \left(1 - |w|^2 \right)^2. \]
\end{proof}

\appendix
\section{Numerical computations}

We perform simulation studies in the R-project \cite{R2021}  to illustrate the properties of the one-step estimators $(Z_n)_n$. 
The version of R is 4.1.1. 
We deal with $n E[ | Z_n - \theta |^2 ]/ \textup{Im}(\theta)^2 = n E[ | Z_n - (\mu+\sigma i) |^2 ]/ \sigma^2$ appearing in Theorem \ref{main-var}. 

For the generators of quasi-arithmetic means, we consider the following four cases that $f_1 (x) = \log x$, $f_2 (x) = \log(x+i)$, $f_3 (x) = 1/(x+i)$ and $f_4 (x) = 1/(x+2i)$. 
For the location and the scale, we consider the following four cases that 
$(\mu, \sigma) = (0,10), (10,1), (0,10), (10,10).$
For the sizes of samples, we consider the following five cases that $n=10, 50, 100, 500, 1000$. 

In each choice of the triplet $(\mu, \sigma, n)$, we compute the average of the values $n|Z_n  - \theta|^2 / \textup{Im}(\theta)^2$ for  $10^6$ samples of size $n$ generated by the {\tt rcauchy()} function in R. 
See  Table \ref{tab:nonshift} for results. 
In this section, we round off the outputs to three decimal places.  

{\small
\begin{table}[H]
\begin{minipage}{0.49\textwidth}
	    \begin{gather*}
		\begin{array}{c|c|c|c|c}
		n      & f_1  & f_2  & f_3  &  f_4   \\
		\hline
		10      & 5.453 & 6.551 & 4.614 & 5.129  \\
		50      & 4.144 & 4.397 & 4.083 & 4.208 \\
		100      &  4.056 & 4.186 & 4.034 & 4.097\\
		500      & 4.016 & 4.043 & 4.013 & 4.025 \\
		1000      &  4.008 & 4.022   & 4.007 & 4.013
		\end{array}\\
		(\mu, \sigma) = (0,1)
		\end{gather*}
	
	\end{minipage}
\begin{minipage}{0.49\textwidth}
	 \begin{gather*}
		\begin{array}{c|c|c|c|c}
		n      & f_1  & f_2  & f_3  &  f_4   \\
		\hline
		10      & 45.014 & 38.179 & 91.065 & 39.535  \\
		50      & 29.718 & 22.231 & 48.523 & 20.440  \\
		100      & 19.742  & 14.915  & 35.411 & 14.386 \\
		500      & 7.420 & 6.393 & 13.280 & 6.572  \\
		1000      & 5.714 & 5.206 & 8.867 & 5.313
		\end{array}\\
		(\mu, \sigma) = (10,1)
	\end{gather*}
	\end{minipage}\\

\begin{minipage}{0.49\textwidth}
	 \begin{gather*}
		\begin{array}{c|c|c|c|c}
		n      & f_1  & f_2  & f_3  &  f_4   \\
		\hline
		10      & 5.463 & 5.463 & 25.377 & 9.934   \\
		50      & 4.138 & 4.152 & 5.672 & 4.239  \\
		100     & 4.063 & 4.072  & 4.571 & 4.076\\
		500     & 4.009 & 4.011 & 4.077 & 4.007 \\
		1000   & 4.005 & 4.006 & 4.037 & 4.003
		\end{array}\\
	(\mu, \sigma) = (0,10)
	\end{gather*}
	\end{minipage}
\begin{minipage}{0.49\textwidth}
	 \begin{gather*}
		\begin{array}{c|c|c|c|c}
		n      & f_1  & f_2  & f_3  &  f_4   \\
		\hline
		10      & 7.066 & 6.795 & 40.737 & 15.413   \\
		50      & 4.491 & 4.442 & 9.545 &  5.266\\
		100    & 4.238 &  4.214 &  6.367 & 4.559\\
		500    & 4.052 & 4.047 & 4.403 & 4.104 \\
		1000  & 4.028 &  4.026 & 4.200 & 4.055
		\end{array}\\
	(\mu, \sigma) = (10,10)
	\end{gather*}
	\end{minipage}\\
	[1\baselineskip]
	\caption{Tables for simulated values for $n E[ | Z_n - (\mu+\sigma i) |^2 ]/ \sigma^2$}
	\label{tab:nonshift}
\end{table}
}

By these numerical computations in Table \ref{tab:nonshift}, we conjecture that 
\[ n E[ | Z_n - \theta |^2 ] \ge 4 \, \textup{Im}(\theta)^2 \]
for every $n$. 
We remark that  the Cram\'er-Rao bound cannot be applied, because $Z_n$ might not be unbiased. 

We now consider the case of the MLE $\hat{\theta}_n$.  
Let $z \in \mathbb{H}$. 
For every $y \in \mathbb R$ and every $t > 0$,  $z$ is the maximal likelihood estimate of $\{x_1, \cdots, x_n\}$ if and only if $tz+x$ is the maximum likelihood estimate of $\{tx_1+y, \cdots, tx_n +y\}$. 
Furthermore, the joint distribution of $(tX_1+y, \cdots, tX_n +y)$ under $P_{y + ti}$ is identical with the joint distribution of $(X_1, \cdots, X_n)$ under $P_{i}$. 
Hence, the distribution of $t \hat{\theta}_n + y$ under $P_i$ is identical with the distribution of $\hat{\theta}_n$ under $P_{ti+y}$.  
Hence, the distribution of $\dfrac{\hat{\theta}_n -\mu}{\sigma}$ is identical with that of $\hat{\theta}_n$ under $\theta = i$. 
Therefore, we can assume the distribution is the standard Cauchy distribution. 

For numerical computations for $n E[ | \hat{\theta}_n - \theta |^2 ]/ \textup{Im}(\theta)^2 = n E[ | \hat{\theta}_n - (\mu+\sigma i) |^2 ]/ \sigma^2$, as in Kravchuk-Pollett \cite{Kravchuk2012}, 
we use the {\tt nlminb()} function in R. 
Here we assume that the real and imaginary parts of the initial point of the algorithm are given by the one-step estimator $Z_n$ associated with $f_3$. 
The results are summarized in Table \ref{tab:MLE}. 

\begin{table}[H]
\[
\begin{array}{c|c}
		n      &   V_n   \\
		\hline
		10      &    5.732 \\
		50      &    4.254 \\
		100      &  4.117\\
		500      & 4.029\\
		1000      & 4.015
		\end{array}
\]\\
\caption{Tables for simulated values for $V_n = n E[ | \hat{\theta}_n - (\mu+\sigma i) |^2 ]/ \sigma^2$}
		\label{tab:MLE}
\end{table}

By Table \ref{tab:MLE}, 
we conjecture that 
\[ n \frac{E[ | \hat{\theta}_n - \theta |^2 ]}{\textup{Im}(\theta)^2} = 4 + O(n^{-1}), \ n \to \infty, \]
which is compatible with the asymptotic expansion for the MLE.

The one-step estimators $(Z_n)_n$ do not have such invariance as the MLE has, so the value of $n E[ | Z_n - \theta|^2 ]/ \textup{Im}(\theta)^2$ depends on the value of $\theta$. 
As in the case that $(\mu,\sigma) = (10,1)$ in Table \ref{tab:nonshift}, if the ratio between the location and the scale is large, then, the performances of the one-step estimators could be bad in particular for samples of small sizes. 
We consider adjusting the median in the definition of $Y_n$ as in \cite{Kravchuk2012}. 

Let $M_n$ be the median of $\{X_1, \cdots, X_n\}$. 
Let 
\[ \widetilde Y_n := M_n + f^{-1} \left( \frac{1}{n} \sum_{j=1}^{n} f\left(X_j - M_n \right)  \right). \]
Let 
\[ \widetilde Z_n :=  \widetilde{Y_{n}} + \frac{2 \textup{Im}(\widetilde{Y_{n}} ) i}{n} \sum_{j=1}^{n} h(X_j, \widetilde{Y_{n}}), \]
which is the one-step estimator of $(\widetilde Y_n)_n$. 
Although we are not sure whether the conclusion of Theorem 4.2 holds or not for $(\widetilde Z_n)_n$, 
we deal with $n E[ | \widetilde Z_n - \theta |^2 ] / \textup{Im}(\theta)^2 = n E[ | \widetilde Z_n - (\mu+\sigma i) |^2 ]/ \sigma^2$. 

{\small 
\begin{table}[H]
\begin{minipage}{0.49\textwidth}
	\begin{gather*}
		\begin{array}{c|c|c|c|c}
		n      & f_1  & f_2  & f_3  &  f_4   \\
		\hline
		10      & 5.884 & 6.669 & 5.954 &  5.621  \\
		50      & 4.279 & 4.429 & 4.168 &  4.247\\
		100      & 4.129 & 4.204 & 4.075 & 4.116\\
		500      & 4.032 & 4.047 & 4.021 & 4.029 \\
		1000    & 4.016 & 4.024  & 4.011 & 4.015		
		\end{array}\\
	(\mu, \sigma) = (0,1)
	\end{gather*}
	\end{minipage}
\begin{minipage}{0.49\textwidth}
	\begin{gather*}
		\begin{array}{c|c|c|c|c}
		n      & f_1  & f_2  & f_3  &  f_4   \\
		\hline
	         10      & 5.891 & 6.665 & 5.904 & 5.654  \\
		50      & 4.284 & 4.436 & 4.173 & 4.253  \\
		100      &  4.134 &  4.211 & 4.079 & 4.120\\
		500      & 4.028 & 4.043 & 4.017 & 4.025 \\
		1000     &  4.009 &  4.017 & 4.004 & 4.008
		\end{array}\\
	(\mu, \sigma) = (10,1)
	\end{gather*}
	\end{minipage}\\

\begin{minipage}{0.48\textwidth}
	\begin{gather*}
		\begin{array}{c|c|c|c|c}
		n      & f_1  & f_2  & f_3  &  f_4   \\
		\hline
		10      & 5.887 & 5.878 & 81.877 & 27.160  \\
		50      & 4.273 & 4.268 & 5.270 & 4.316 \\
		100      & 4.137 &  4.134 & 4.439 & 4.148\\
		500      & 4.025 & 4.024 & 4.091 & 4.028 \\
		1000    & 4.013 & 4.013  & 4.047 & 4.014
		\end{array}\\
	(\mu, \sigma) = (0,10)
	\end{gather*}
	\end{minipage} \, \, 
\begin{minipage}{0.48\textwidth}
	\begin{gather*}
		\begin{array}{c|c|c|c|c}
		n      & f_1  & f_2  & f_3  &  f_4   \\
		\hline
		10      &  5.881 & 5.872 & 81.176 & 27.085 \\
		50      & 4.276 & 4.271 & 5.264 & 4.315 \\
		100      &  4.136 &  4.133 & 4.438 & 4.149\\
		500      & 4.030 & 4.030 & 4.096 &  4.033 \\
		1000    & 4.019 & 4.018  & 4.053 & 4.020
		\end{array}\\
	(\mu, \sigma) = (10,10)
	\end{gather*}
	\end{minipage}\\
	[1\baselineskip]
	\caption{Tables for simulated values for $n E[ | \widetilde Z_n - (\mu+\sigma i) |^2 ]/ \sigma^2$}
	\label{tab:shift}
\end{table}}

By these numerical computations, 
when we consider the median-adjusting,  
the logarithmic functions $f(x) = \log(x+\alpha)$, which are $f_1$ and $f_2$, are better than the M\"obius transformations $f(x) = 1/(x+\alpha)$, which are $f_3$ and $f_4$,  as the generators of the quasi-arithmetic means. 
Since $(M_n - \mu)/\sigma$ is the median of $\{ (X_j - \mu)/\sigma\}_j$, 
\[ \frac{\widetilde Y_n-\mu}{\sigma} = \frac{M_n -\mu}{\sigma} + f_1^{-1}\left( \frac{1}{n} \sum_{j=1}^{n} f_1 \left( \frac{X_j -\mu}{\sigma} - \frac{M_n -\mu}{\sigma} \right) \right)  \]
and
\[ \frac{\widetilde Z_n -\mu}{\sigma} = \frac{\widetilde Y_n -\mu}{\sigma} + \frac{2 \textup{Im}\left((\widetilde Y_n-\mu)/\sigma\right) i}{n} \sum_{j=1}^{n} h\left(\frac{X_j -\mu}{\sigma}, \frac{\widetilde Y_n -\mu}{\sigma} \right), \]
the distribution of $\dfrac{\widetilde Z_n -\mu}{\sigma}$ is identical with that of $\widetilde Z_n$ under $\theta = i$. 
Therefore, for $f_1$, $(\widetilde Z_n)_n$ has such invariance as the MLE has, and hence, it suffices to consider the standard Cauchy distribution only. 
Furthermore, the performances of $(\widetilde Z_n)_n$ in the case of the logarithmic functions $f_1$ and $f_2$ are similar to that of the case of MLE in Table \ref{tab:MLE}. 

However, there are no theoretical guarantees of these performances. 
We are not sure whether $(\widetilde Y_n)_n$ is consistent or not, and, $\widetilde Y_n$ for $f_1$ works well only if the sample size $n$ is even. 
If $n$ is odd, then, $\textup{Im}(\widetilde Y_n) = 0$.  
It would be overcome by changing the definition of the median slightly. 
One way is to adopt $(x_{(n-1)/2}+x_{(n+1)/2}+x_{(n+3)/2})/3$ for $\{x_1 < \cdots < x_{n}\}$. 
Here we do not discuss this issue further. 
See also  \cite[Subsection 5.7]{Akaoka2021-3} for some delicate issues for median-adjusting.

\subsection*{\it Acknowledgements} \ 
The authors wish to express their gratitude to referees for their many helpful comments and suggestions to improve the paper, in particular, for suggesting numerical computations in Appendix and notifying the authors of a recent reference \cite{Kato2020} concerning the circular Cauchy distribution. 
The second and third authors were supported by JSPS KAKENHI 19K14549 and 16K05196 respectively.

\bibliographystyle{amsplain}
\bibliography{Cauchy2}

\end{document}